\documentclass[12pt, notitlepage]{article}
\makeindex
\usepackage{amsfonts,amssymb,amsmath} 
\usepackage{latexsym} 

\makeatletter

\newcommand{\be}{\begin{equation}}
\newcommand{\ee}{\end{equation}}
\newcommand{\beano}{\begin{eqn*}}
	\newcommand{\eeano}{\end{eqnarray*}}
\newcommand{\ba}{\begin{array}}
	\newcommand{\ea}{\end{array}}

\newcommand{\vtwo}{\vskip 4ex}

\def\eop{\hbox{{\vrule height7pt width3pt depth0pt}}} 
\newenvironment{proof}
{\begin{sloppypar}\vtwo\noindent{\bf Proof~: }}
	{\hspace*{\fill}\eop
	\end{sloppypar}} 
	\newcommand\bprf{\begin{proof}}  
		\newcommand\eprf{\end{proof}}

	\newtheorem{theorem}{Theorem}[section]
	\newtheorem{proposition}[theorem]{Proposition}
	\newtheorem{corollary}[theorem]{Corollary}

	\newtheorem{remark}[theorem]{Remark}
	\newtheorem{example}[theorem]{Example}

	\def\disc{{\rm disc}}
	\def\Nrd{{\rm Nrd}}

	\title{Adjoint groups over ${\mathbb Q}_p (X)$ and R-equivalence - revisited}
	
	\author{R. Preeti, A. Soman}

	\date{}

\usepackage{multirow}
\usepackage{array}
\newcolumntype{P}[1]{>{\centering\arraybackslash}p{#1}}
\newcolumntype{M}[1]{>{\centering\arraybackslash}m{#1}}
	
\begin{document}
		\maketitle

	
	\begin{abstract}
		
	\end{abstract}
	 We obtain a class of examples of non-rational adjoint classical groups of type $^2A_n$ and a group of type $^2D_3$ over the function field $F$ of a smooth geometrically integral curve over a $p$-adic field with $p \neq 2$. We also show that for any group of type $C_n$ over $F$, 
	 the group of rational equivalence classes of $G$ over $F$ is trivial, i.e., $G(F)/R=(1)$.

\section{Introduction} 
Let $F$ be the function field of a smooth, geometrically integral curve over a $p$-adic field with $p\neq 2$ and $G$ be an absolutely simple adjoint algebraic group over $F$. In \cite{PS0} we show that if $G$ is 
an adjoint algebraic group over $F$ of type ${\ }^2A_n^*$, $C_n$ or 
$D_n$ ($D_4$ non-trialitarian) such that the associated hermitian form 
has even rank, trivial discriminant (if $G$ is of type ${\ }^2A_n^*$ or $D_n$) 
and trivial Clifford invariant (if $G$ is of type $D_n$) then the group of 
rational equivalence classes, $G(F)/R$ is trivial. 
In this paper we show that the hypotheses on the hermitian forms associated to $G$ are necessary for groups of type $A_n$ and we extend the result in (Theorem $6.1$, \cite{PS0}) to any group of type $C_n$. Further, for a group $G$ of type $D_3$ with 
$h$ being an associated hermitian form, we show that if $\disc (h) $ is non-trivial then $G(F) /R$ need not be trivial. For general groups of outer type
$A_n$ and $^1 D_n$, the triviality of $G(F) /R$ remains open.  
The main results in this paper are: 

\begin{theorem}\label{A} 
  Let $p$ be a prime such that $p \not= 2$. Let $F =
  {\mathbb Q}_p (t)$ be the rational function field in one variable over
  the $p$-adic field ${\mathbb Q}_p$. 
  Then for every positive integer $n \geq 2$, there exist
  absolutely simple adjoint algebraic groups $G$ of type $^2A_{2n-1}$ over $F$ such that, the group of rational equivalence classes over $F$ is non-trivial, i.e., $G(F)/R\neq (1)$. 
\end{theorem}
\medskip

\begin{theorem}\label{cor-C} 
  Let $F$ be the function field of a smooth, geometrically integral curve over a $p$-adic field with $p\neq 2$. For any absolutely simple adjoint classical group $G$ of type $C_n$, the group of rational equivalence classes over $F$ is trivial, i.e., $G(F)/R = (1)$. 

\end{theorem}
\medskip

\begin{remark} \label{D} Using the exceptional isomorphisms for algebraic groups of low rank, a
group of type $^2A_3$ can be identified with a group of type $^2D_3$.
Hence for $p \neq 2$, by Theorem \ref{A} we know that there exists a group $G$ of type $^2D_3$ over $\mathbb{Q}_p (t)$ such that $G({\mathbb{Q}_p (t)}) / R \neq (1)$. 
However in Example \ref{DD} we present a direct construction of such a group
of type $D_3$ over $\mathbb{Q}_p (t)$ which has non-trivial R-equivalence classes. 
\end{remark} 
\medskip

\medskip 
The triviality of $G(F)/R$ is closely related to $G$ being rational over $F$. 
For $F$ as above we note that the cohomological dimension of $F$ is $\leq 3$ (see \cite{Se1}).
Over some fields of cohomological dimension $\leq 3$, 
Merkurjev (Theorem $3$, \cite{M}) has shown that there exist groups
of type $^2 D_3$ which are non-rational and 
Gille (Theorem $3$, \cite{G1}) has shown that there exist groups of
type $^1 D_4$ which are non-rational. 
However, those examples 
do not yield information on the triviality of $G(F)/R$ over the function field $F$ of a $p$-adic curve. Furthermore, 
the existence of such non-rational groups over the function field of a $p$-adic curve
$F$, was not known earlier. 
As an immediate corollary of Theorem \ref{A} and remark \ref{D} we
get examples of non-rational adjoint classical groups over $F$.  
\medskip 

\begin{corollary}
  Let $F=\mathbb{Q}_p(t)$ with $p\neq 2$, be a rational function field over $\mathbb{Q}_p$ in one variable.
\begin{enumerate} 
\item If $p \not= 2$, for every positive integer
  $n \geq 2$, there exist absolutely simple adjoint algebraic groups $G$ of type $^2A_{2n-1}$  defined over $F$ which are non-rational. 
\item For $ p \neq 2$, there exist groups $G$ of type $^2D_3$ defined over $F$ which are non-rational. 
\end{enumerate}  
\end{corollary}

\medskip

We summarise below the known results, including the ones in this paper, alongwith the remaining open cases for convenience. Let $F$ be the function field of a smooth, geometrically integral curve over a $p$-adic field with $p\neq 2$.
	\medskip
	
\begin{center}
	\centering
\begin{tabular}{ |c|M{13cm}| }
	\hline
	{\bf Type of group $G$} &  \\ \hline
	\hline
	\multirow{1}{*}{$^1A_{n-1}$, $(n\geq 2)$} & \vspace{.5mm}
        $\textbf{G}(F)/R = (1)$ and $G$ is rational (\S $2$, \cite{M}).\\ \hline
	\multirow{10}{*}{$^2A_{n-1}$} & \vspace{.5mm} $\textbf{G}(F)/R = (1) $; when $n$ is odd due to Merkurjev (\S $2$, \cite{M}).\\\cline{2-2}
	& \vspace{.5mm}  $\textbf{G}(F)/R = (1) $; when associated central simple algebra has square-free index and associated hermitian form has even rank and trivial discriminant (Theorem $5.3$, \cite{PS0}). \\ \cline{2-2}
   & \vspace{.5mm} $\textbf{G}(F)/R$ need not be $(1)$ ; when $m \geq 2$ and  $n= 2m$ and $p \not= 2$ (Theorem \ref{A}) in this paper. \\ \cline{2-2} 
	 & \vspace{.5mm}  Triviality of $\textbf{G}(F)/R$ is not known when underlying central simple algebra has index divisible by square of a prime.  \\\hline
	 
	\multirow{1}{*}{$B_n$} & \vspace{.5mm} $\textbf{G}(F)/R = (1)$ and
        $G$ is rational (\S $2$, \cite{M}). \\ \hline
	\multirow{1}{*}{$C_n$} & $\textbf{G}(F)/R = (1) $, for $n=2$ and $n$ odd by Merkurjev (\S $2$, lemma $3$, \cite{M}) and for other cases (Theorem $6.1$, \cite{PS0}) and Theorem \ref{cor-C} in this paper.\\ \hline 
	\multirow{8}{*}{$^1D_n$} & \vspace{.5mm} $\textbf{G}(F)/R = (1)$; when $n=2$ ($D_4$ non-trialitarian) and $n=3$ due to Merkurjev (proposition $5$, \cite{M}).\\ \cline{2-2}
	& \vspace{.5mm} $\textbf{G}(F)/R = (1) $; when associated hermitian form has even rank, trivial discriminant and trivial Clifford invariant (Theorem $7.2$, \cite{PS0}). \\ \cline{2-2}
	& \vspace{.5mm} Not known when associated hermitian form has even rank, trivial discriminant and non-trivial Clifford invariant.  \\ \hline
	\multirow{1}{*}{$^2D_n$} & $\textbf{G}(F)/R$ need not be $ (1) $;  when $n=3$
                 (Theorem \ref{D} this paper). \\ \hline
	
\end{tabular}
\end{center}
\medskip

{\bf Acknowledgement:} We thank J.-P. Tignol for his comments on an earlier version of this preprint, especially on the proof of Theorem \ref{A} in the case when
$-1$ is a square in $\mathbb{Q}_p$.  
	
	\section{Preliminaries}

	In this section, we 
	recall some basic notions on hermitian forms over algebras with involutions. 
	
	\subsection{Notations and basic definitions} 
	
	Let $K$ be a field of characteristic different from $2$. 
	Let $(A, \sigma)$ denote a central simple algebra over a field $K$ with 
	an involution, that is, $\sigma :A \to A$ is an anti-automorphism 
	of order $2$. Let $E = K^{\sigma}$ denote the fixed field of $K$ under $\sigma$. 
	Then either $E = K$ or $K$ is a quadratic field extension of $E$. 
	
	\medskip 
	
	For $\epsilon = \pm 1$, we denote by $(V,h)$ an $\epsilon$-hermitian form over $(A,\sigma)$.  
	We denote by $W(A, \sigma)$ 
	the Witt group of non-degenerate hermitian forms over $(A, \sigma)$. 
	We refer to \cite{L} and 
	\cite{S} for basic facts on quadratic and 
	hermitian forms and their Witt groups. If $A= E$ and $\sigma$ is the 
	identity then $W(A, \sigma) = W(E)$, the usual Witt group of 
	non-degenerate quadratic forms over $E$. For $a,b\in E^*$, we denote the associated quaternion algebra over $E$ by $(a,b)$. This is a central simple algebra over $E$ of degree $2$ with basis $\{1,i,j,ij\}$ subject to the relations $i^2=a, j^2=b, ij=-ji$. We denote an $n$-fold Pfister form by $\langle\langle a_1,\ldots,a_n\rangle\rangle=\langle 1,-a_1\rangle\otimes\cdots\otimes\langle1,-a_n\rangle$, for $a_1,\ldots,a_n\in E^*$.
	\medskip 
	
	\subsection{Invariants of hermitian forms over $(A, \sigma)$} 
	
	We recall some Galois cohomological invariants of hermitian forms over 
	$(A, \sigma)$. 
	We refer the reader to \cite{BP1}, \cite{BP2}, \cite{PP} and \cite{P} 
	for more details. 
	Suppose that $A = D$, a division algebra. Let $(V,h)$ be a non-degenerate 
	hermitian form over $(D, \sigma)$.

	\begin{enumerate} 
		\item Rank $rank(h)$ : The {\it rank} of $(V,h)$ 
		is defined as the dimension of the 
		underlying $D$-vector space $V$, say $n = dim_D (V)$. 
		
		\item Discriminant $\disc(h)$ : Given a basis of the $D$-module $V$, the 
		hermitian form $h$ is given by some matrix $M(h)$ in this basis. Let 
		$A = M_n(D)$ and let $m^2 = dim_K (A)$. 
		The {\it discriminant}, $\disc (h)$ of $(V,h)$ is 
		defined as: \\
		$\disc(h) = (-1)^{m(m-1)/2}\ \Nrd_A (M(h)) \in E^*/E^{*2}\  
		{\rm if}\ \sigma\ {\rm is\ of\ first\ kind\ }$ \\ 
		$\disc(h) =  (-1)^{m(m-1)/2}\ \Nrd_A (M(h)) \in E^*/N_{K|E} (K^*)\  
		{\rm if}\ \sigma\ {\rm is\ of\ second\ kind\ }$. 
		
		\item Clifford invariant $c(h)$ : Suppose that $(D,\sigma)$ is of the 
		first kind and orthogonal type over $E$. Let $(V,h)$ and $(V,h^{\prime})$ be 
		non-degenerate hermitian forms over $(D, \sigma)$ such that $\disc(h) = 
		\disc(h^{\prime} )$. The {\it relative Clifford invariant}, 
		$Cl_h (h^{\prime} ) \in {}_2 Br (E) / (D) $ is defined by Bartels \cite{B}. Let 
		$H_{2n}$ be a hyperbolic form of rank $2n$ over $(D,\sigma)$. Let $(V,h)$ 
		be a hermitian form such that $rank (h) = 2n$ and $\disc(h)$ is trivial. Then 
		the {\it Clifford invariant}, $c(h) : = Cl_{H_{2n}} (h)$, 
		(see \S $2$ \cite{BP1} for more details). If $D = E$ then this invariant 
		is the usual Clifford invariant of the quadratic form $h$. 
		
		\item Rost invariant $R(h)$ : We refer to the relevant sections in $\cite{BP1}$ and $\cite{PP}$ for the definition of this invariant.
\end{enumerate}	
		\medskip

	\medskip 
	
	\subsection{Multipliers of similitudes}\label{similitude}
	
	Let $K$ be a field of characteristic not equal to $2$ and $(A, \sigma)$ be a central simple algebra over $K$ with an involution. Let $E=K^{\sigma}$ be the fixed subfield under $\sigma$ in $K$. 
	A similitude of $(A,\sigma)$ is an element $g\in A$ such that $\sigma(g)g\in E^*$. The scalar $\sigma(g)g$ is called the multiplier of the similitude $g$ and it is denoted by $\mu(g)$. The set of all similitudes of $(A,\sigma)$ is a subgroup of $A^*$ which is denoted by $\textrm{Sim}(A,\sigma)$, and the map $\mu$ is a group homomorphism $\mu: \textrm{Sim}(A,\sigma)\rightarrow E^*$. 
We refer to (\S $12$.B and \S $12$.C \cite{KMRT}) for more details on similitudes of 
	$(A, \sigma)$. 	
	The image of the map $\mu$ is denoted by $G(A,\sigma)$. 
	Suppose $\sigma$ is an involution of orthogonal type 
	on a central simple algebra $A$ of even degree $2m$ 
	over a field $K$. For every similitude $g\in\textrm{Sim}(A,\sigma)$ we have $\textrm{Nrd}_A(g)=\pm\mu(g)^m$. A similitude is called proper if $\textrm{Nrd}_A(g)=+\mu(g)^m$, otherwise it is called an improper similitude. In this case we write $G_+(A,\sigma)$ for the group of multipliers corresponding to proper similitudes. 
	By convention, for a symplectic involution or 
	an involution of the second kind, we 
	set $G_+(A,\sigma)=G(A,\sigma)$. Consider the algebraic group $\textrm{\textbf{PSim}}(A,\sigma)$ defined by
	\begin{displaymath}
	\textrm{\textbf{PSim}}(A,\sigma)= \textrm{\textbf{Sim}}(A,\sigma)/R_{K/E}(G_m), 
	\end{displaymath}
	where $R_{K/E}$ is the Weil restriction from $K$ to $E$ and $K$ is the center of $A$ (see \cite{KMRT} for details). 
	The connected component of identity of the algebraic group $\textrm{\textbf{PSim}}(A,\sigma)$ is denoted by $\textrm{\textbf{PSim}}_+(A,\sigma)$.
	Following the usual notation (\S $12$.B \cite{KMRT}), we denote 
	the algebraic groups $\textrm{\textbf{PSim}}(A,\sigma)$ according to the type 
	of $\sigma$ as : 
	
	\[
	\textrm{\textbf{PSim}}(A,\sigma)=
	\begin{cases}
	\textrm{\textbf{PGO}}(A,\sigma) &\text{ if $\sigma$ is of orthogonal type,}\\
	\textrm{\textbf{PGSp}}(A,\sigma) &\text{ if $\sigma$ is of symplectic type,}\\
	\textrm{\textbf{PGU}}(A,\sigma) &\text{ if $\sigma$ is of unitary type.}
	\end{cases}
	\]
	We consider the groups $\textrm{Sim}(A,\sigma)$ and $\textrm{PSim}(A,\sigma)$ as the group of $F$-points of the corresponding algebraic groups $\textrm{\textbf{Sim}}(A,\sigma)$ and $\textrm{\textbf{PSim}}(A,\sigma)$ respectively. When the involution $\sigma$ is of unitary or symplectic type $\textbf{\textrm{PSim}}(A,\sigma)$ is a connected group (see \S $2$ \cite{M}). In the case of an orthogonal involution $\sigma$ we denote the connected component of $\textbf{\textrm{PSim}}(A,\sigma)$ by $\textbf{\textrm{PGO}}_+(A,\sigma)$. 
	\medskip 
	
	For $(A, \sigma)$, a central simple algebra over $K$ with an involution, let $E=K^{\sigma}$. Set 
	$NK^*:=\{\sigma(z)z:z\in K^*\}$. If $\sigma$ is an involution of the first kind 
	then $NK^*=E^{*2}$ and if $\sigma$ is of the second kind then, $NK^*$ is the group of norms $N_{K/E}(K^*)$ of the  quadratic field 
	extension $K/E$. We denote by $\textrm{Hyp}(A,\sigma)$ the subgroup of $E^*$ generated by the norms of all finite extensions $M/E$ such that $\sigma_M$ is a hyperbolic involution. 
	Further, for $(D, \sigma)$ 
	a central division algebra over $K$ with an involution $\sigma$, and 
	$h$ a non-degenerate hermitian form over $(D,\sigma)$ of rank $m$, we denote 
	by $G(h)$ and $G_+(h)$ the groups $G(M_m(D),\sigma_h)$ and $G_+(M_m(D),\sigma_h)$ respectively, where $\sigma_h$ is the adjoint involution corresponding to 
	$h$. We also 
	denote by $\textrm{Hyp}(h)$ the group $\textrm{Hyp}(M_m(D),\sigma_h)$. 
	\medskip
	
	\section{Some known results} 
	
	In this subsection we recall some results which are used in the proofs of the main theorems. We refer to the earlier section for 
	notation and terminology. We start with the following 
	result due to Merkurjev. 
	\medskip 
	
	\begin{theorem}\label{Merkurjev}(Theorem 1, \cite{M}) 
		With notation as in section \ref{similitude} above, there is a natural isomorphism
		\begin{displaymath}
		\textrm{\textbf{PSim}}_+(A,\sigma)(E)/R \simeq G_+(A,\sigma)/(NK^*\cdot\textrm{Hyp}(A,\sigma)).
		\end{displaymath}
	\end{theorem}
	\medskip

	Let $F$ be the function field of a smooth, geometrically integral curve over a $p$-adic field with $p\neq 2$. 
	We next recall the well-known result due to Parimala-Suresh on the
        $u$-invariant, 
	$u(F)$ of $F$, where we recall that for a field $E$, 
	$u(E)$ is the largest dimension of an 
	anisotropic quadratic form over $E$ and is $\infty$ if such an integer does 
	not exist for $E$. 
	
	\begin{theorem} (Theorem $4.6$, \cite{PS}) 
		Let $F$ be the function field of a curve over a $p$-adic field. 
		If $p\neq 2$, then $u(F)=8$. 
	\end{theorem} 
	\medskip 
	
	\medskip 
	
We now list adjoint classical groups over $F$ for which $G(F)/R$ is known to be trivial (see \cite{PS0}). 
We start with the following theorem on groups of type $A_n$ (Theorem $5.3$, \cite{PS0}).
\medskip

\begin{theorem} Let $F$ be the function field of a smooth, geometrically integral curve over a $p$-adic field with $p\neq 2$. Let $Z/F$ be a quadratic 
	field extension and $(Q,\tau)$ a quaternion algebra over $Z$ with a unitary $Z/F$ involution. Let $h$ be a hermitian form over $(Q,\tau)$ of even rank $2n$ and trivial discriminant. Then for the adjoint group $\textbf{\textrm{PGU}}(M_{2n}(Q),\tau_h)$, the 
	group of rational equivalence classes is trivial, i.e., 
	$$ 
	\textbf{\textrm{PGU}}(M_{2n}(Q),\tau_h)(F)/R = (1). 
	$$
	In fact, 
	$$
	G(h)=\textrm{Hyp}(h)=F^*.
	$$ 
\end{theorem} 
\medskip 

For groups of type $C_n$ we have the following theorem (Theorem $6.1$, \cite{PS0}). 
\medskip 

\begin{theorem} Let $F$ be the function field of a smooth, geometrically integral curve over a $p$-adic field with $p\neq 2$. Let $(A,\sigma)$ be a central simple algebra over $F$ with a symplectic involution. Let $h$ be a hermitian form over $(A,\sigma)$ of even rank $2n$. Then for the adjoint group 
	$\textbf{\textrm{PGSp}}(M_{2n}(A),\sigma_h)$, the group of rational equivalence classes is trivial, i.e., 
	$$ 
	\textbf{\textrm{PGSp}}(M_{2n}(A),\sigma_h)(F)/R = (1). 
	$$ 
	In fact 
	$$ 
	G(h)=\textrm{Hyp}(h)=F^*. 
	$$ 
\end{theorem}
\medskip 

For groups of type $D_n$, we have the following theorem (Theorem $7.2$, \cite{PS0}). 
\medskip

\begin{theorem}\label{Ortho} Let $F$ be the function field of a 
	smooth, geometrically integral curve over a $p$-adic field with $p\neq 2$. Let $(A,\sigma)$ be a central simple algebra over $F$ with an orthogonal involution. Let $h$ be a hermitian form over $(A,\sigma)$ of even rank $2n$, trivial discriminant and trivial Clifford invariant. Then for the adjoint group 
	$\textbf{\textrm{PGO}}_+(M_{2n}(A),\sigma_h)$, 
	the group of rational equivalence classes is trivial, i.e., 
	$$ 
	\textbf{\textrm{PGO}}_+(M_{2n}(A),\sigma_h)(F)/R = (1). 
	$$ 
	In fact, 
	$$ 
	G_+(h)=\textrm{Hyp}(h)=F^*. 
	$$ 
\end{theorem}
\medskip 
	
\bigskip

	\section{Adjoint groups of type ${\ }^2A_n$} 
	
	Let $F$ be the function field of a smooth, geometrically integral curve over a $p$-adic field with $p\neq 2$. In this section we show that the hypothesis of Theorem $5.3$, \cite{PS0} is necessary. 
	For every positive integer $n \geq 2$ we give examples of
        absolutely simple adjoint algebraic groups $G$ of type $^2A_{2n-1}$ over the rational function field $F$ over ${\mathbb Q}_p$ with $p \neq 2$ such that the group of rational equivalence classes over $F$
        is non-trivial, i.e., $G(F)/R\neq (1)$.
	
	For a central simple algebra $(B,\tau)$ of even degree with an unitary $Z/F$ involution, let $D(B,\tau)$ denote its \emph{discriminant algebra} over $F$. Recall that $D(B,\tau)$ has a canonical involution of first kind  (see \S 10, \cite{KMRT}) for details on the discriminant algebra. We start with the following proposition. 
	\medskip 
	
	\begin{proposition}\label{disc-hyp}
		Let $F$ be the function field of a smooth, geometrically integral curve over a $p$-adic field with $p\neq 2$. Let $Z/F$ be a quadratic field extension and $(Q,\tau)$ a quaternion algebra over $Z$ with unitary $Z/F$ involution. 
		 Let $h$ be a non-degenerate  hermitian form of even rank $2n$ over $(Q,\tau)$. Then for the associated discriminant algebra $D(M_{2n}(Q),\tau_h)$ we have 
		\[
		\Nrd(D(M_{2n}(Q),\tau_h))=\textrm {Hyp}(h).
		\]
	\end{proposition}
	\begin{proof}
		Let $d\in F^*$ be such that $Z=F(\sqrt{d})$. We claim that $D(M_{2n}(Q),\tau_h)\sim (d,\disc(h))_F$. If $Q$ is split then by (corollary $10.35$, \cite{KMRT}) we have $D(M_{2n}(Q),\tau_h)\sim (d,\disc(h))_F$. If 
		$Q$ is non-split let $K=F(R_{Z/F}(SB(M_{2n}(Q))))$ be the function field  of the Weil transfer of the Severi-Brauer variety of $M_{2n}(Q)$. Then $M_{2n} (Q)_K$ is split over $K$. Therefore $D(M_{2n}(Q),\tau_h)_K \sim (d,\disc(h))_K$. 
	As the map $Br(F)\rightarrow Br(K)$ is injective (see corollary $2.12$, \cite{MT}), we have $D(M_{2n}(Q),\tau_h)\sim (d,\disc(h))_F$.

	To prove the inclusion
        $\Nrd(D(M_{2n} (Q), \tau_h)) \subset
                        \textrm {Hyp}(h)$, we use results of \cite{PS0}. 
                        Let $L/F$ be a finite field extension such that $D(M_{2n}(Q),\tau_h)$ is split over $L$. Then $(d,\disc(h))$ splits over $L$, which implies $\disc(h)\in N_{LZ/L}((LZ)^*)$. Thus, over $L$, 
		$h$ has even rank and trivial discriminant in $L^*/N_{LZ/L}((LZ)^*)$. By Theorem $5.3$, \cite{PS0}, $\textrm{Hyp}(h_{L})=L^*$. Hence $N_{L/F}(L^*)=N_{L/F}(\textrm{Hyp}(h_L))\subset \textrm{Hyp}(h)$. Thus, $\Nrd (D(M_{2n}(Q),\tau_h))\subset \textrm{Hyp}(h)$. 
		\medskip

                Conversely, if $h$ is hyperbolic over a finite field extension $M$ of $F$ then, $h_M$ has trivial discriminant. Therefore,  $D(M_{2n}(Q),\tau_h) \sim (d,\disc(h))$ splits over $M$. Hence, $\textrm{Hyp}(h)\subset\Nrd(D(M_{2n}(Q),\tau_h))$. This inclusion has been proved in \cite{BMT}.

	\end{proof} 
	\medskip
	\smallskip 
\noindent 
	{\bf Proof of Theorem \ref{A} : } 

Let $F=\mathbb{Q}_p(t)$ with $p\neq 2$. Let $b\in\mathbb{Z}_p^*$ be a non-square unit. Then $\langle\langle b,p\rangle\rangle$ is an anisotropic $2$-fold Pfister form over $\mathbb{Q}_p$ (see VI, $2.2$, \cite{L}).
Let $H=(b,t)_F$ be a quaternion algebra over $F$. 
\medskip

Let $K=\mathbb{Q}_p(\sqrt{p})$. As $K$ is a totally ramified field extension of $\mathbb{Q}_p$, the residue field of $K$ is the same as the residue field of $\mathbb{Q}_p$. Hence by Hensel's lemma, $b\notin K^{*2}$.
By taking residues in the Laurent series field $K((t))$ and using
($1.9$, chapter VI, \cite{L}) we see that $H$ does not split over
$F({\sqrt{p}})$, i.e., the norm form of $H$, $n_H = \langle 1, -t,-b,tb\rangle$ is anisotropic over $F({\sqrt{p}})$. 

As $b$ is a unit in $\mathbb{Z}_p^* $, by (chapter VI, $2.5$, \cite{L}),  
  $(-1,b)$ splits over $\mathbb{Q}_p$. Hence $\langle 1,1\rangle\cdot n_H=0$ in the Witt group $W(F)$ of $F$. Hence $-1\in\Nrd(H)$ (cf. chapter III, $2.3$ and $2.4'$, \cite{L}). Let $u\in H^*$ be such that $-1=\Nrd(u)$. We have the following two cases.
\medskip

\noindent{\bf case.1.} $n=2m$. We choose a quaternion basis $1,i,j,ij$ of $H$ such that $i$ commutes with $u$. So $i^2 = b$, $j^2 = t$, $i \cdot j = -j \cdot i$ and $i \cdot u = u \cdot i$. Now consider the involution 
\[\sigma=\textrm{Int}\begin{pmatrix}
j/t & 0 & \dots  & 0 \\
0   & i/b& \dots  & 0 \\
\vdots & \vdots   & \ddots \\
0 & 0 & \dots  & i/b
\end{pmatrix} \circ ^-t
\]
on $M_{2m}(H)$, where $^-$ denotes the canonical involution on $H$ and $t$ denotes the transpose of a matrix.
\medskip

As $\sigma (\textrm{diag}(j,i,\cdots, i))
= -\textrm{diag}(j,i, \cdots, i)$, $\sigma$ is an orthogonal involution and \[\disc(\sigma)=\Nrd(j)\cdot\Nrd(i)\equiv t\cdot b \not\equiv 1 ~(\textrm{mod }F^{*2}). \]
\medskip

Let $(B,\tau)=(M_{2m}(H)\otimes_F F(\sqrt{p}),\sigma\otimes\gamma)$, where $\sigma$ is the orthogonal involution constructed above and $\gamma$ is the non-trivial automorphism of $F(\sqrt{p})/F$. So $(B,\tau)$ is a central simple algebra over $F(\sqrt{p})$ 
with an unitary $F(\sqrt{p})/F$ involution $\tau$. 
Let $g=\textrm{diag}(j,uj, \cdots, uj)$, where $u\in H^*$ is the element chosen above. As $i \cdot u = u \cdot i$, we have $\sigma(g)=\textrm{diag}(-j, j\bar{u}, \cdots, j\bar{u})$. Hence 
\[ \mu(g)= \sigma(g)\cdot g = -t .\]
Thus $-t\in G(B,\tau)$. 
\medskip

By (proposition $10.33$, \cite{KMRT}), as $D(B,\tau) \sim (p,\disc(\sigma))_F=(p,t\cdot b)_F$ we have

\begin{align*}
\mu(g)\cup D(B,\tau)& = (-t)\cup (p)\cup(t\cdot b) & \text{in } H^3(F,\mu_2)\\
&= (-t)\cup (p)\cup (b) &\text{since } \langle\langle -t,t\cdot b\rangle\rangle\simeq\langle\langle -t,b\rangle\rangle\\
&= (t)\cup (p)\cup (b) & \text{since } (-1,b)_F \text{ is split} 
\end{align*}

We claim that $(t)\cup(p)\cup(b)\neq 0\in H^3(F,\mu_2)$. If $(t)\cup(p)\cup(b)=0\in H^3(F,\mu_2)$ then $\langle 1,-t\rangle\cdot\langle\langle p,b\rangle\rangle=0\in I^3(F)$. But by taking residues in $\mathbb{Q}_p ((t))$ and noting
that $\langle \langle b,p\rangle\rangle$ is anisotropic over $\mathbb{Q}_p$
we have $\langle 1, -t\rangle \cdot \langle \langle p,b\rangle\rangle
\not= 0$. 
Thus, $(t)\cup (p)\cup (b)\neq 0$. Hence $\mu(g)\notin\Nrd (D(B,\tau))$ (see chapter III, $2.3$ and $2.4'$, \cite{L}). By proposition \ref{disc-hyp} above, $\mu(g)\notin Hyp(B,\tau)$. Hence by (\S $2$, \cite{M}), $\textrm{\textbf{PGU}}(B,\tau)(F)/R\neq (1)$.

\medskip

\noindent{\bf case.2.} $n=2m+1$. We choose a quaternion basis $1,i,j,ij$ of $H$ such that $j$ commutes with $u$. So $i^2 = b$, $j^2 = t$, $i \cdot j
= -j \cdot i$ and $j \cdot u = u \cdot j$. 
Now consider the involution $\sigma$,

\[\sigma=\textrm{Int}\begin{pmatrix}
i & 0 & \dots  & 0 \\
0   & j& \dots  & 0 \\ 
\vdots & \vdots   & \ddots \\
0 & 0 & \dots  & j
\end{pmatrix} \circ ^-t
\]
on $M_{2m+1}(H)$ with $m\geq 1$, where $^-$ denotes the canonical involution on $H$ and $t$ denotes the transpose of a matrix.
\medskip

As $\sigma (\textrm{diag}(i,j,\cdots, j))
        = -\textrm{diag}(i,j, \cdots, j)$, $\sigma$ is an orthogonal involution and $\disc(\sigma)=(-1)^{2m+1}\cdot\Nrd(i)\cdot\Nrd(j)^{2m}\equiv b ~(\textrm{mod }F^{*2})$.
	\medskip
	
	Let $(B,\tau)=(M_{2m+1}(H)\otimes_F F(\sqrt{p}),\sigma\otimes\gamma)$, where $\sigma$ is the orthogonal involution constructed above and $\gamma$ is the non-trivial automorphism of $F(\sqrt{p})/F$. So $(B,\tau)$ is a central simple algebra over $F(\sqrt{p})$ 
	with an unitary $F(\sqrt{p})/F$ involution $\tau$. 
	Let $g=\textrm{diag}(j,j\bar{u}, \cdots, j\bar{u})$, where $u\in H^*$ is the element chosen above. As $\sigma(g)=\textrm{diag}(j, -ju, \cdots, -ju)$, we have  
        \[ \mu(g)= \sigma(g)\cdot g = t. \] 
	Thus $t\in G(B,\tau)$. 
	\medskip
	
	Further, by (proposition $10.33$, \cite{KMRT}), 
	\begin{align*}
	D(B,\tau) &\sim \lambda^{2m+1} M_{2m+1}(H)\otimes_F(p,\disc(\sigma))_F\\
	& \sim H\otimes_F (p,b)\\
	& \sim (b,t)\otimes_F (p,b)\\
	& \sim (b,t\cdot p).
	\end{align*} 

Therefore, we have
\begin{align*}
\mu(g)\cup D(B,\tau)& = (t)\cup (b)\cup(t\cdot p) & \text{in } H^3(F,\mu_2)\\
&= (t)\cup (b)\cup (-p) &\text{since } \langle\langle t,t\cdot p\rangle\rangle\simeq\langle\langle t,-p\rangle\rangle\\
&= (t)\cup (b)\cup (p) & \text{since } (-1,b)_F \text{ is split} 
\end{align*}

Hence arguing exactly as in case.1.), we have,
$\mu (g) \cup D(B, \tau) \not= 0$ in
$H^3(F, \mu_2)$. 
Hence $\mu(g)\notin\Nrd (D(B,\tau))$ (see chapter III, $2.3$ and $2.4'$, \cite{L}). By (lemma $10$, \cite{BMT}), $\mu(g)\notin Hyp(B,\tau)$. Hence by (\S $2$, \cite{M}), $\textrm{\textbf{PGU}}(B,\tau)(F)/R\neq (1)$.
	\medskip
	
	In view of the above (Theorem \ref{A}) we get the following result.
	
	\begin{corollary}
		Let $F$ be the rational function field in one variable over a $p$-adic field with $p \not= 2$. For each positive integer $n \geq 2$, there exists an absolutely simple adjoint algebraic group of type $^2A_{2n-1}$ defined over $F$ which is non-rational. 
	\end{corollary} 
	\begin{proof}
    For $F$ as in the corollary, let $(B,\tau)$ be a central simple algebra constructed as in the Theorem \ref{A} above. We have seen in Theorem \ref{A}, that the group of rational equivalence classes of $\textrm{\textbf{PGU}}(B,\tau)(F)/R$ is non-trivial. Hence the group $\textrm{\textbf{PGU}}(B,\tau)$ is not rational (proposition $1$, \cite{M}).
    \end{proof}

	\medskip
	
	\section{Adjoint groups of type $C_n$}
	Let $F$ be the function field of a smooth, geometrically integral curve over a $p$-adic field with $p\neq 2$. In this section we extend (Theorem $6.1$, \cite{PS0}) to an arbitraty group of type $C_n$ over $F$. We start with the following. 
	
	\begin{theorem}\label{C}
		Let $F$ be the function field of a smooth, geometrically integral curve over a $p$-adic field with $p\neq 2$. Let $(Q_1\otimes_F Q_2,\sigma)$ be a biquaternion division algebra over $F$ with symplectic involution. Let $h$ be a hermitian form over $(Q_1\otimes_F Q_2,\sigma)$ of odd rank $r$. Then for the adjoint group $\textrm{\textbf{PGSp}}(M_r(Q_1\otimes Q_2),\sigma_h)$, the group of rational equivalence classes is trivial, i.e., 
		\[
		\textrm{\textbf{PGSp}}(M_r(Q_1\otimes_F Q_2),\sigma_h)(F)/R = (1).
		\]
		In fact, \[G(h)=\textrm{Hyp}(h)=F^*.\]
	\end{theorem}
	\begin{proof}
		Let $q_A$ be an Albert form for $Q_1\otimes_F Q_2$. So $q_A$ is a $6$-dimensional quadratic form over $F$ with trivial discriminant. As the $u$-invariant of $F$, $u(F)=8$ (see Theorem $4.6$, \cite{PS}), the group of spinor norms of $q_A$, $Sn(q_A)=F^*$ (see proof of the Theorem $4.1$, \cite{PS0}). Let $\lambda\in F^*$. Then $\lambda\in Sn(q_A)$. Thus there exists a finite field extension $L/F$ such that $q_A$ is isotropic over $L$ and $\lambda=N_{L/F}(x)$, for some $x\in L^*$. Over $L$, $Q_1\otimes Q_2\sim H$ for some quaternion algebra. By Morita correspondence $M_r(Q_1\otimes Q_2,\sigma_h)$ will correspond to ($M_{2r}(H),\sigma_{h_L})$. By (Theorem $6.1$, \cite{PS0}), $\textrm{Hyp}(h_L)=L^*$ and thus $x\in\textrm{Hyp}(h_L)$. Therefore, $\lambda=N_{L/F}(x)\in N_{L/F}(\textrm{Hyp}(h_L))\subset\textrm{Hyp}(h)$. Hence $F^*=\textrm{Hyp}(h)=G(h)$. Thus, the group of rational equivalence classes, $\textrm{\textbf{PGSp}}(M_r(Q_1\otimes_F Q_2),\sigma_h)(F)/R=(1)$.
	\end{proof}
	\medskip

        \noindent
        {\bf Proof of Theorem \ref{cor-C} :} 
	  Let $G$ be an absolutely simple adjoint algebraic group of type $C_n$ over $F$. By Weil's classification results (\cite{We}), the group $G$ is associated to a central simple algebra $A$ over $F$ with a symplectic involution. Thus $\textrm{exp}(A)\leq 2$ and by (corollary $2.2$, \cite{PS3}) its index, $\textrm{ind}(A)\leq 4$. The corollary follows by combining the above Theorem \ref{C} along with
          Theorem $6.1$ and corollary $6.2$ from \cite{PS0}. 

	\medskip
	
	\section{Adjoint groups of type $D_n$} 
	
	Let $F$ be the function field of a smooth, geometrically integral curve over a $p$-adic field with $p\neq 2$.
	For an adjoint classical group $G$ of type $D_n$ over $F$ we consider separately the cases when the associated hermitian form $h$ has trivial discriminant (that is, $G$ is of type $^1 D_n$) and $h$ has non-trivial discriminant (that is, 
	$G$ is of type $^2D_n$). 
	
	\subsection{Adjoint groups of type $^2D_n$}
	
	Let $F$ be the function field of a smooth, geometrically integral curve over a $p$-adic field with $p\neq 2$. 
	In this section we give an example of an adjoint group $G$ of type $^2D_3$ over a field $F$ for which the group of R-equivalence classes is not trivial, as mentioned in remark \ref{D} . Hence such a group $G$ is non-rational over $F$.  
	\medskip

        \begin{example} \label{DD} 
		Let $F=\mathbb{Q}_p(t)$ be a rational function field in one variable over $\mathbb{Q}_p$ with $p\neq 2$, where $t$ is an indeterminate. Let $p$ denote a uniformizing parameter of $\mathbb{Q}_p$ and $(p,u)$ be the unique quaternion division algebra over $\mathbb{Q}_p$. Let $Q_1=(p,u)\otimes_{\mathbb{Q}_p}F$, $Q_2=(t,u)$ and $Q=Q_1\otimes_F Q_2= (p\cdot t, u)$ be quaternion algebras over $F$. Let $~\bar{}~$ denote the canonical involution on $Q$. 
		Let $\sigma_{h'}$ be the adjoint involution on $M_2(Q)$ corresponding to the skew-hermitian form $h' := \langle 1,-p\rangle\cdot\langle j\rangle = \langle j, -pj\rangle$ over $(Q, \bar{}~)$. In other words, 
		$(M_2(Q), \sigma_{h'})
                = (M_2(F), \sigma_{\langle 1,-p \rangle}) \otimes
                (Q, \sigma_{\langle j \rangle})$, where $\sigma_{\langle 1,-p\rangle}$ and $\sigma_{\langle j\rangle}$ are the adjoint involutions corresponding to $\langle 1,-p\rangle$ and $\langle j\rangle$. 
		Note that $\sigma_{\langle 1,-p\rangle}$ and $\sigma_{\langle j\rangle}$ are orthogonal involutions on $M_2(F)$ and $Q$ respectively. Moreover, $\disc(\sigma_{\langle 1,-p\rangle})=p\in F^*/F^{*2}$ and $\disc(\sigma_{\langle j\rangle})= u\in F^*/F^{*2}$. Therefore, one of the components of the Clifford algebra $C(M_2(Q),\sigma_{h'})$ is  Brauer-equivalent to $(\disc(\sigma_{\langle 1,-p\rangle}),\disc(\sigma_{\langle j\rangle}))=(p,u)$ and the other one to $Q\otimes_F (p,u)\sim (t,u)$, (see Tao's result, page $150$, \cite{KMRT}).
		\medskip 
		
		Let $h :=  h' + <i> = \langle j,-pj,i\rangle$ be a skew-hermitian form over $(Q, \bar{}~)$ and let $\sigma_h$ be the corresponding adjoint involution. Thus, we have $\disc(\sigma_h)=p\cdot t$. Set $L=F(\sqrt{p\cdot t})$. 
		We show that the element $-p\cdot t$ is a non-trivial element in the  group $\textrm{\textbf{PSim}}_+(M_3(Q),\sigma_h)(F)/R$. Clearly $-p\cdot t$ belongs to $N_{L/F}(L^*)$. As $(-1, u)$ is split over $\mathbb{Q}_p$ we can find in $Q$ a pure quaternion $i^{\prime}$ with $i^{\prime 2} = -pt$ that anticommutes with $j$. Then $g = diag (i^{\prime}, i^{\prime}, i )$ is a similitude with multiplier $\mu(g) = -pt$. Hence $-pt \in G_+ (M_3(Q), \sigma_h)$. 
		\medskip
		
		Observe that $(-p\cdot t)\cup (p)\cup (u)= (t)\cup (p)\cup (u)$ and $(-p\cdot t)\cup (t)\cup (u)=(t)\cup (p)\cup (u)$. Hence $(-p\cdot t)\cup Q_i= (t)\cup (p)\cup (u)$, for $i=1,2$. We claim that $(t)\cup (p)\cup (u)\neq 0\in H^3(F,\mu_2)$. 
		Consider the Pfister form $q=\langle 1,-t\rangle\cdot\langle\langle p,u\rangle\rangle$ corresponding to the symbol $(t)\cup (p)\cup (u)$. Here we use the notation $\langle\langle a\rangle\rangle= \langle 1,-a\rangle$ for $a\in F^*$. We write $q=\langle\langle p,u\rangle\rangle \perp\langle -t\rangle\cdot\langle\langle p,u\rangle\rangle$. We consider the quadratic form $q$ in the field of formal Laurent series $\mathbb{Q}_p((t))$ with uniformizing parameter $t$. By (VI, $1.9$ (2), \cite{L}), $q$ is anisotropic over $\mathbb{Q}_p((t))$ as $\langle\langle p,u\rangle\rangle$ is anisotropic over $\mathbb{Q}_p$. Hence, $(-p\cdot t)\cup\Nrd(Q_i)=(t)\cup (p)\cup (u)\neq 0\in H^3(F,\mu_2)$, for $i=1,2$. By (chapter III, $2.3$ and $2.4'$, \cite{L}), $-p\cdot t\notin\Nrd(Q_i)$. Thus, $-p\cdot t$ is a non-trivial element in the group $\textrm{\textbf{PSim}}_+(M_3(Q),\sigma_h)(F)/R$ (by proposition $9$ \cite{M}). 
\end{example} 	
	\medskip 

        As an immediate consequence we have :

	\begin{theorem} 
		Let $F$ be the rational function field in one variable over a $p$-adic field with $p\neq 2$. Then there exists an absolutely simple adjoint algebraic group of type $^2D_3$ defined over $F$ which is non-rational. 
	\end{theorem} 
	\begin{proof}
    For $F$ as in the theorem, consider the central simple algebra with involution $(M_3 (Q), \sigma_h)$ constructed as in the example \ref{DD} above. We have seen in example \ref{DD}, that the group of rational equivalence classes of $\textrm{\textbf{PSim}}_+(M_3 (Q),\sigma_h)(F)/R$ is non-trivial. Hence the group $\textrm{\textbf{PSim}}_+(M_3(Q), \sigma_h)$ is not rational (proposition $1$, \cite{M}).
    \end{proof}

	\bigskip
	\medskip 
	
	\subsection{Groups of type $^1D_n$} 
	
	Let $F$ be the function field of a smooth, geometrically integral curve over a $p$-adic field with $p\neq 2$. Let $G$ be an absolutely simple adjoint algebraic group of inner type $^1D_n$ defined over $F$.
        Let $(A, \sigma)$ be a central simple algebra with orthogonal involution over $F$ associated to the group $G$. The group $G$ being of type $^1D_n$ translates in to $(A, \sigma)$ having even rank and trivial discriminant. 

        If we assume further that $(A, \sigma)$ has
        trivial Clifford invariant then the group of rational equivalence classes, $G(F)/R=(1)$ by Theorem $7.2$, \cite{PS0}. Combining this with the results
        in this paper leaves only one case open
        where the behaviour of $G(F)/R$ is not known, namely when $(A,\sigma)$ has even rank, trivial discriminant and \emph{non-trivial} Clifford invariant.

\medskip 

\medskip 

\noindent 
Preeti Raman \\ 
Department of Mathematics \\
Indian Institute of Technology (Bombay) \\ 
Powai, Mumbai-400076, India\\ 
{\tt{preeti@math.iitb.ac.in}} \\
\medskip

\noindent
Abhay Soman \\
Department of Mathematical Sciences \\
Indian Institute of Science Education and Research, Mohali \\
Sector 81, SAS Nagar, Manauli, Punjab-140306, India \\ 
\tt{somanabhay@iisermohali.ac.in}
\end{document}